\documentclass[12pt, reqno]{amsart}

\usepackage{enumerate}
 
 \usepackage{geometry }
 \usepackage{float}

\usepackage{amssymb,amsmath,amscd}
\usepackage{amsthm}
 \numberwithin{dummy}{section}

\usepackage{varioref}
\theoremstyle{plain}
\usepackage{color}

 \numberwithin{equation}{section}

\usepackage{amssymb,amsmath,amscd}
\usepackage{amsthm}
\newtheorem{theorem}{Theorem}[section]
\newtheorem{theorem*}{Theorem }
\newtheorem{proposition}{Proposition}[section]

\newtheorem{proposition*}{Proposition A\!\!}
\newtheorem{corollary*}{Corollary A\!\!}
\newtheorem{lemma}{Lemma}[section]

\newtheorem{cor}{Corollary}[section]

 \usepackage{mathtools}
 \usepackage{lipsum}
\usepackage[bbgreekl]{mathbbol}

\DeclareSymbolFontAlphabet{\mathbb}{AMSb}
\DeclareSymbolFontAlphabet{\mathbbl}{bbold}
 
 \usepackage[
 pdftex=true, colorlinks=true,  citecolor=cyan,   
 citebordercolor={0 .255 .255} , 
  linkbordercolor={0 .255 .255},  
  linkcolor=cyan
 ]{hyperref}

   \usepackage{cite}

 \def\equationautorefname~#1\null{(#1)\null}
 
\usepackage{multirow}
 \usepackage{stmaryrd}
 \usepackage{graphicx}
 \usepackage{caption}
 \usepackage{booktabs}
 \usepackage{threeparttable}

\DeclareMathOperator{\Det}{\mathrm{Det}}

\renewcommand{\det}{\mathbf{det}}

\makeatletter
\ifcsname phantomsection\endcsname
    \newcommand*{\qrr@gobblenexttocentry}[5]{}
\else
    \newcommand*{\qrr@gobblenexttocentry}[4]{}
\fi
\newcommand*{\addsubsection}{%
    \addtocontents{toc}{\protect\qrr@gobblenexttocentry}%
    \subsection}
\makeatother

\usepackage{lipsum}
\usepackage{mathabx}
\usepackage[normalem]{ulem}

\date{ August 20th, 2020}

      \usepackage[normalem]{ulem}
      \usepackage{color}
      \usepackage{cancel}
      \usepackage{slashed}

\newcommand{\set}[2]{ \left\{\,#1\; | \; #2\,\right\} }

\newcommand{\nc}{\newcommand}
\nc{\ep}{\varepsilon}
\nc{\iu}{i}
\nc{\integer}{\mathbb{Z}}
\nc{\real}{\mathbb{R}}
\nc{\complex}{\mathbb{C}}
\nc{\Cone}{\mathcal{C}}
 \DeclareMathOperator{\GL}{GL}
 \DeclareMathOperator{\Sym}{Sym}
 \DeclareMathOperator{\Sp}{Sp}
 
\begin{document}
\title{The compression semigroup of the dual Vinberg cone}
\author{Hideyuki Ishi}
\address{ H. Ishi\\ Osaka City University, 3-3-138 Sugimoto-cho, Sumiyoshi-ku, Osaka 558-8585, Japan}
\email{hideyuki@sci.osaka-cu.ac.jp}
\author{and Khalid Koufany }
 \address{ K. Koufany\\ Institut Elie Cartan de Lorraine, UMR CNRS 7502, University of Lorraine,  F-54506 Vandoeuvre-lès-Nancy, France}
\email{khalid.koufany@univ-lorraine.fr}

\date{August 20, 2020}

 \maketitle
 
 {\centering\footnotesize Dedicated to the memory of Professor Takaaki Nomura.\par}
 
\begin{abstract} 
We investigate the semigroup associated to the dual Vinberg cone and prove its triple and Ol'shanski\u{\i} polar decompositions.
 Moreover, we show that the semigroup does not have the contraction property with respect to the canonical Riemannian metric on the cone.
\end{abstract}
 
%
%
%
%
 \tableofcontents
 
\section{Introduction and preliminaries}
Semigroups of transformations leaving invariant a given set is a well known tool in  various fields of mathematics, 
 for example invariant convex cone theory and geometric control theory.  
In Lie group setting, 
 probably the most important compression semigroups come from the Ol'shanski\u{\i} semigroups, 
 i.~e. compression semigroups of symmetric spaces $G_\mathbb{C}/G$, 
 where $G$ is a Hermitian Lie group. 
One extremely useful structure property of such semigroups is 
 the existence and uniqueness of the Ol'shanski\u{\i} polar decomposition $G\exp(iC)$,
 where $C$ is a convex cone in the Lie algebra of $G$ which is invariant under the adjoint action of $G$. 
This decomposition has many applications to representations theory, 
 see for example \cite{Olsh82, N, H-N}. 

A compression semigroup
 associated naturally to an Euclidean Jordan algebra $E$ was introduced in \cite{Kou95}. 
It is
 the compression semigroup of a symmetric cone $\Cone$ 
 (the open cone of invertible squares in $E$), 
 $\Gamma := \{g\in Co(E)\; |\; g(\Cone)\subset \Cone\}$, 
 where $Co(E)$ is the conformal group of $E$.
This semigroup $\Gamma$ satisfies the Ol'shanski\u{\i} polar decomposition
 and, in addition, $\Gamma$ admits a triple decomposition. 
Furthermore, elements of $\Gamma$ are proved to be contractions 
 for the invariant Riemannian metric on $\Cone$ \cite{Kou95, Kou02}
 and also for the Hilbert metric \cite{Kou06} and the Finsler metric \cite{Lim00}. 
The contraction property has many applications, 
 for example in Kalman Filtering theory (for the Hamiltonian semigroup) \cite{Boug93}.

The purpose of this article is to study 
 the compression semigroup of a homogeneous non-symmetric convex cone, 
 which gives a new example of Lie semigroup which admits 
 both the Ol'shanski\u{\i} polar decomposition and a triple decomposition,
 but does not have the contraction property with respect to the canonical metric.
More precisely,
the homogeneous cone $\Omega$ is given by
$$\Omega := \left\{x \in \real^5\; | \; x_1>0,\, x_2>0,\, x_1x_2x_3-x_1x_5^2-x_2x_4^2>0\right\},
$$
 which is called the {\em dual Vinberg cone}
 (\cite{Vin60, Vin63}).  
 
Let us first summarize some well known facts about the real symplectic group and the symplectic semigroup,
 which will be utilized frequently in the investigation of the cone $\Omega$.
Let $\mathrm{Sym}(3,\mathbb{R})$ denote the space of $3\times 3$ real symmetric matrices,
 and  $\mathrm{Sym}^{+}(3,\mathbb{R})$ (resp. $\mathrm{Sym}^{++}(3,\mathbb{R})$)
 the subset of positive (resp. positive definite) matrices.
Then $\mathrm{Sym}^{++}(3,\mathbb{R})$ is a symmetric cone 
 in the Euclidean Jordan algebra $\mathrm{Sym}(3,\mathbb{R})$ 
 with the inner product given by
 $(x|y) = \mathrm{tr}\,xy$. 
Denote by $\Delta_1, \Delta_2$ and $\Delta_3$ 
 the principal minors of matrices in $\mathrm{Sym}(3,\mathbb{R})$.
 For a matrix $M$, denote by $M^T$ its transpose, 
 and if $M$ is invertible, $M^{-T}$ will denote $(M^T)^{-1}$.
 
Recall the symplectic group $\Sp(6,\mathbb{R})=\{g\in \GL(6,\mathbb{R}) \; | \; gJg^T=J\}$ with $J=\begin{pmatrix} 0 & -I\\ I& 0\end{pmatrix}$. 
In a block form,  an element 
 $g=\begin{pmatrix} A & B \\ C & D \end{pmatrix}\in \GL(6,\mathbb{R})$ 
 with $A, B, C, D \in \mathrm{Mat}(3, \real)$
 belongs to  $\Sp(6,\mathbb{R})$ if and only if
\begin{equation} \label{eqn:Sp6}
\begin{aligned}
 & A^TC\, , D^TB \in \Sym(3,\mathbb{R}),\\
  &D^TA - B^TC = I,
\end{aligned}
\end{equation}
or equivalently
\begin{equation} \label{eqn:Sp6-2}
\begin{aligned}
 & B A^T, C D^T \in \Sym(3,\mathbb{R}),\\
 &A D^T - B C^T = I.
\end{aligned}
\end{equation}
\begin{lemma}\label{lemma:HCdecomp_Sp6}
An element 
 $g=\begin{pmatrix} A&B\\ C&D\end{pmatrix}\in Sp(6,\mathbb{R})$ 
 has a unique triple decomposition
\begin{equation}\label{eqn:HC_Sp6}
 \begin{pmatrix} A&B\\ C&D\end{pmatrix}
 = \begin{pmatrix} I&v\\ 0&I\end{pmatrix} 
 \begin{pmatrix} L&0\\ 0&L^{-T}\end{pmatrix} 
 \begin{pmatrix} I&0\\ u& I \end{pmatrix}
\end{equation}
 with $u, v\in \Sym(3,\mathbb{R})$ and $L\in GL(3,\mathbb{R})$ 
 if and only if $D$ is invertible, and in this case
 \begin{equation}\label{eqn:Luv}
 L=D^{-T}=A-BD^{-1}C, \, v=BD^{-1} \text{ and }\, u=D^{-1}C.
 \end{equation}
 \end{lemma}
 
It is well known that the symplectic group $\Sp(6,\mathbb{R})$ acts on   
 the Siegel upper half space 
 $T_{\Sym^{++}(3,\mathbb{R})} := \Sym(3,\mathbb{R})+i\Sym^{++}(3,\mathbb{R})$  
 by linear fractional transformations,
 that is,
$$
g\cdot z= (Az+B)(Cz+D)^{-1},\;\; \text{ where }\; g=\begin{pmatrix} A & B \\ C & D \end{pmatrix}\in \Sp(6,\mathbb{R}), z\in T_{\Sym^{{++}}(3,\mathbb{R})}$$
which induces an isomorphism from $\Sp(6,\mathbb{R})/\{\pm I\}$ onto the holomorphic automorphism group $G(T_{\Sym^{{++}}(3,\mathbb{R})})$ of $T_{\Sym^{{++}}(3,\mathbb{R})}$.
Since $\mathrm{Sym}(3, \real)$ is the \v{S}ilov boundary of $T_{\Sym^{++}(3, \mathbb{R})}$,
 the action of $\Sp(6, \real)$ is extended on $\mathrm{Sym}(3, \real)$
(precisely, one should consider a conformal compacitification of $\mathrm{Sym}(3, \real)$
 on which the actions of all the elements $g \in \Sp(6, \real)$ are well-defined \cite{Kan91}). 
%
In this action 
 we consider the compression semigroup  (called also the symplectic semigroup) 
 of the symmetric cone $\Sym^{++}(3,\mathbb{R})$,
\begin{equation*}
\Gamma_{\Sp}
 := \left\{g\in \Sp(6,\mathbb{R})\; |\; 
 g\cdot \Sym^{++}(3,\mathbb{R})\subset \Sym^{++}(3,\mathbb{R})\right\}
\end{equation*}
which is a closed subsemigroup of $\Sp(6,\mathbb{R})$.

It was proved in \cite{Kou95} that $\Gamma_{\Sp}$ can be given by
\begin{equation}\label{symp-semi}
\Gamma_{\Sp}=\left\{\begin{pmatrix}A&B\\ C&D \end{pmatrix}\in \Sp(6,\mathbb{R})\; | \; 
 D \in \GL(3,\mathbb{R}), \; CD^T, D^TB\in \Sym^{+}(3,\mathbb{R}) \right\}
 \end{equation}
and has a triple decomposition
$\Gamma_{\Sp}=\Gamma_{\Sp}^+  G(3,\mathbb{R}) \Gamma_{\Sp}^-$,
where
\begin{eqnarray*}
\Gamma_{\Sp}^+& := &\left\{\begin{pmatrix}I&B\\ 0&I \end{pmatrix} \; | \;     B\in \Sym^{++}(3,\mathbb{R}) \right\},\\
  G(3, \real) & := &
 \set{\begin{pmatrix} A & 0 \\ 0 & A^{-T} \end{pmatrix} }
           {A \in \mathrm{GL}(3,\real) },\\
\Gamma_{\Sp}^-& := &\left\{\begin{pmatrix}I&0\\ C&I \end{pmatrix} \; | \;     C\in \Sym^{++}(3,\mathbb{R}) \right\}.
\end{eqnarray*}
It was also proved that the symplectic semigroup satisfies
 the following Ol'shanski\u{\i} polar decomposition 
 $\Gamma_{\Sp}=G(3,\mathbb{R})\exp(C_{\Sp})$,
 where $C_{\Sp}$ is the closed convex cone
$$C_{\Sp} := \left\{\begin{pmatrix}0 & B\\ C& 0\end{pmatrix}\in \mathfrak{sp}(6,\mathbb{R}) \; |\; B, C\in\Sym^+(3,\mathbb{C})\right\},$$
 and
 $\mathfrak{sp}(6,\mathbb{R})$ is the Lie algebra of $\Sp(6,\mathbb{R})$,
 that is,
\begin{equation*} 
\begin{aligned}
\mathfrak{sp}(6,\mathbb{R})&=\{X\in M(3,\mathbb{R}) \; |\;  XJ+JX^T=0\}\\
&= \left\{ \begin{pmatrix} A & B \\ C& -A^T  \end{pmatrix}\; |\; 
  A \in \mathrm{Mat}(3, \mathbb{R}),\,\,B, C\in  \Sym(3,\mathbb{R}) \right\}. \\
\end{aligned}
\end{equation*}


Now we turn to 
 the dual Vinberg cone $\Omega$.
Let $V$ be the subspace of $\mathrm{Sym}(3,\mathbb{R})$ defined by
$$
 V :=
 \left\{x=\begin{pmatrix} x_1 & 0 & x_4 \\ 0 & x_2 & x_5 \\ x_4 & x_5 & x_3 \end{pmatrix},\; 
        x_1, \dots, x_5 \in \mathbb{R} \right\}.
$$
Then 
 $\Omega$ is naturally identified with
 the intersection $\mathrm{Sym}^{++}(3,\mathbb{R}) \cap V$,
 that is, 
$$
\begin{array}{rl}
 \Omega &= \{x \in V \; |\; \Delta_1(x)>0, \Delta_2(x)>0, \Delta_3(x)>0 \}\\
 &= \{x \in V \; |\; x \mbox{ is positive definite} \}.
 \end{array}
$$
Let $T_\Omega := V+i\Omega\subset V_\mathbb{C}$ 
 be the tube domain over $\Omega$,
$G(T_{\Omega})$  
 the identity component of the holomorphic automorphism group
 on the tube domain $T_{\Omega}$, 
 and $\Gamma$ the compression semigroup  
\begin{equation} \label{eqn:def_of_Gamma}
 \Gamma := \left\{g \in G(T_{\Omega})\; |\;  g\cdot\Omega \subset \Omega\right\}
\end{equation}
  of $\Omega$.
 This semigroup $\Gamma$ is a main object of the present work.

 Here we explain the organization of this paper.
In Section 2,
 we describe the group $G(T_{\Omega})$ as a subgroup of $\Sp(6, \real)$.
Then we give a characterization of $\Gamma$ as a subset of $G(T_{\Omega})$
 using the triple decomposition in Section 3.
In Section 4, we show that $\Gamma$ also admits an Ol'shanski\u{\i} polar decomposition.
Finally, in Section 5, we show that $\Gamma$ does not have a contraction property 
 with respect to the canonical Riemannian metric on $\Omega$.

%
%
%
%
\section{The holomorphic automorphism group of $T_\Omega$}
First we shall determine the linear automorphism group 
$$G(\Omega) :=\{g\in GL(V) \; | \; g\Omega = \Omega\}$$
 of the cone $\Omega$.
Define
 $$ 
 H := \left\{A = \begin{pmatrix} a_1 & & \\ 0 & a_2 & \\ a_4 & a_5 & a_3 \end{pmatrix}\; |\; a_1, \dots, a_5 \in \mathbb{R},\,\,a_1 a_2 \ne 0,\,a_3>0\right\}, 
$$
 and let $H^+$ be the subset of $H$ consisting of   diagonal matrices with positive entries.    
Then $H$ forms a Lie group,
 and $H^+$ is its identity component.
Let $\rho : H \to GL(V)$ be the representation of $H$
 given by $\rho(A)x := A x A^T\,\,\,(A \in H,\,x \in V)$.
Then $H^+$ as well as $H$ acts transitively on the cone $\Omega \subset V$ by $\rho$.
In other words, we have $\rho(H) \subset G(\Omega)$
 and $\Omega = \rho(H^+) I_3$.
For a parameter $(s_1, s_2, s_3) \in \mathbb{C}^3$,
 let $\Delta_{(s_1,s_2, s_3)}$ be the function on $\Omega$ given by
\begin{align*}
 \Delta_{(s_1,s_2, s_3)}(x) &:= \Delta_1(x)^{s_1 - s_2} \Delta_2(x)^{s_2 - s_3} \Delta_3(x)^{s_3}\\
 &= x_1^{s_1-s_3} x_2^{s_2 - s_3} (\det x)^{s_3} 
\quad (x \in \Omega). 
\end{align*}
The function $\Delta_{(s_1,s_2, s_3)}$ is relatively invariant under the action of $H^+$: 
\begin{equation} \label{eqn:Det-invariance}
 \Delta_{(s_1,s_2, s_3)}(\rho(A)x) = a_1^{2s_1} a_2^{2s_2} a_3^{2 s_3} \Delta_{(s_1,s_2, s_3)}(x)
 \quad (x \in \Omega,\, A \in H^+).
\end{equation}
Indeed, this equality characterizes the function $\Delta_{(s_1,s_2, s_3)}$ up to a constant multiple.

Let $\Omega^* \subset V$ be the dual cone of $\Omega$.
Namely,
 $\Omega^* := \left\{\xi \in V\; |\; (x|\xi) >0,\,  \forall x \in \overline{\Omega} \setminus \{0\}\right\}$.
The so-called {\em K\"ocher-Vinberg characteristic function} $\varphi_{\Omega}$ of $\Omega$ is defined by 
$\varphi_{\Omega}(x) := \int_{\Omega^*} e^{-  (x|\xi) }\,d\xi$ for $x \in \Omega$.
It is known (see \cite[Proposition I.3.1]{F-K}) that, for any $g \in G(\Omega)$,
 we have 
 
\begin{equation} \label{eqn:phi-invariance}
 \varphi_{\Omega}(gx) = |\Det\, g|^{-1} \varphi_{\Omega}(x) \quad (x \in \Omega). 
\end{equation}
For $A = \mathrm{diag}(a_1, a_2, a_3) \in H$ and $x \in V$,
 we observe that
$$ 
\rho(A)x 
 = \begin{pmatrix} a_1^2 x_1 & 0 & a_1 a_3 x_4 \\ 
 0 & a_2^2 x_2 & a_2 a_3 x_5 \\ 
 a_1 a_3 x_4 & a_2 a_3 x_5 & a_3^2 x_3 \end{pmatrix},$$
 so that $\Det\, \rho(A) = a_1^3 a_2^3 a_3^4$.
For a general $A \in H^+$, because of the factorization $A = \mathrm{diag}(a_1, a_2, a_3) A'$
 with a unipotent $A' \in H^+$,
 we have again $\Det\, \rho(A) = a_1^3 a_2^3 a_3^4$.
Therefore,
 comparing \autoref{eqn:Det-invariance} and \autoref{eqn:phi-invariance}, 
 we conclude that
 there exists a constant $C >0$ for which
\begin{equation} \label{eqn:phi-det}
 \varphi_{\Omega}(x) = C \Delta_{(-3/2, -3/2, -2)}(x) 
 = C x_1^{1/2} x_2^{1/2} (\det x)^{-2}.
\end{equation} 

Let $G(\Omega)_{I_3}$ be the isotropy subgroup of $G(\Omega)$ at $I_3 \in \Omega$,
 and take $g \in G(\Omega)_{I_3}$.
In general, for a function $F$ on $\Omega$, we denote by $g^*F$ the pullback $F \circ g$.
Since $G(\Omega)_{I_3}$ is a compact group, 
 we have $|\Det\,g| = 1$, 
 so that 
 $g^* \varphi_{\Omega}^2 = \varphi_{\Omega}^2$
 thanks to \autoref{eqn:phi-invariance}.
Thus, by the uniqueness of irreducible factorization of the rational function $\varphi_{\Omega}^2$,
 we have
\begin{equation} \label{eqn:case_I}
 g^*x_1 = C_1 x_1, \quad
   g^*x_2 = C_2 x_2, \quad
   g^*\det x = C_3 \det x
\end{equation}
 or
\begin{equation} \label{eqn:case_II}
 g^*x_1 = C_1 x_2, \quad
   g^*x_2 = C_2 x_1, \quad
   g^*\det x = C_3 \det x
\end{equation}
with $C_1 C_2 C_3 = 1$. 
On the other hand, since $g\cdot I_3 = I_3$, we have $C_1 = C_2 = C_3 = 1$.
Let us consider the case \autoref{eqn:case_I}.
We have $g^* \det x = \det x$, which means that
\begin{equation} \label{eqn:comparison-det}
 x_1 x_2 (g^* x_3) - x_1 (g^* x_5)^2 - x_2 (g^* x_4)^2 = x_1 x_2 x_3 - x_1 x_5^2 - x_2 x_4^2. 
\end{equation}
From this equality,
 we deduce $g^* x_5 = \pm x_5 + \alpha x_2$ with some $\alpha \in \mathbb{R}$.
In fact, if $g^*x_5$ would contains other terms, for instance $\gamma x_3$, 
 then the left-hand side should contain the term of $x_1 x_3 x_5$,
 which does not appear in the right-hand side.
By the same argument, we have $g^* x_4 = \pm x_4 + \beta x_1$ with some $\beta \in \mathbb{R}$.
Actually, we have \autoref{eqn:comparison-det} in this case with
$$ g^* x_3 = x_3 + \beta^2 x_1 + \alpha^2 x_2 \pm 2 \beta x_4 \pm 2 \alpha x_5.$$
On the other hand, since $g\cdot I_3 = I_3$, we have $\alpha = \beta = 0$. 
Therefore we conclude that $g = \rho(\mathrm{diag}(\pm 1, \pm 1, 1 ))$.

Let us turn to the case \autoref{eqn:case_II}. 
Put $\sigma = \begin{pmatrix} 0 & 1 & 0 \\ 1 & 0 & 0 \\ 0 & 0 & 1 \end{pmatrix}$.  
Then $\rho(\sigma) : x = (x_1, \dots, x_5) \mapsto (x_2, x_1, x_3, x_5, x_4)$
 belongs to $G(\Omega)_{I_3}$
 satisfying \autoref{eqn:case_II}.
Furthermore, if $g \in G(\Omega)_{I_3}$ satisfies \autoref{eqn:case_II},
 then $g \circ \rho(\sigma)$ satisfies \autoref{eqn:case_I}.
Now we conclude that
\begin{lemma}
The isotropy subgroup $G(\Omega)_{I_3}$ is a finite group of order 8 generated by
$\rho(\mathrm{diag}(-1,1,1))$, $\rho(\mathrm{diag}(1,-1,1))$, and $\rho(\sigma)$.
\end{lemma}
\begin{cor}
One has $G(\Omega) = \rho(H^+) \rtimes G(\Omega)_{I_3}$.
\end{cor}

We  extend the action of  $G(\Omega)$ to $T_\Omega$ by $g(z)=g(x)+ig(y)$, $z=x+iy$.
The translation  $t_v : z\mapsto z+v$ by $v\in V$ is
 a holomorphic automorphism of $T_\Omega$,
 and the group $N^+$ of all such translations is an Abelian group isomorphic to $V$.
The rational map $s$ on $T_\Omega$ defined by
$$
 s : T_{\Omega} \owns z \mapsto 
 \begin{pmatrix}
 -\frac{1}{z_1} & 0 & \frac{z_4}{z_1} \\ 
 0 & -\frac{1}{z_2} & \frac{z_5}{z_2} \\ 
\frac{z_4}{z_1} & \frac{z_5}{z_2} & \frac{{\det z} }{z_1z_2}
 \end{pmatrix}
 \in T_{\Omega} 
$$
 belongs to $G(T_{\Omega})$.
Note that $s^2 {= \rho(\mathrm{diag}(-1,-1,1))} \not=\mathrm{Id}$, 
 so that $s$ is not an involution, but $s^4 = Id$.
Let $V'$ be the subspace of $V$ defined by
$$
 V' := 
 \left\{u=\begin{pmatrix} u_1 & 0 & 0\\ 0 & u_2 & 0 \\ 0 & 0 & 0\end{pmatrix}\; |\; 
     u_1, u_2 \in \mathbb{R}\right\}.
$$
 For any $u\in V'$,
 let $\tilde{t}_u=s\circ t_u\circ s$  and denote by $N^-$ the subgroup of $G(T_\Omega)$  of these transformations.
 
Keeping in mind the inclusion $T_{\Omega} \subset T_{\Sym^{++}(3,\mathbb{R})}$,
 we shall realize the group $G(T_{\Omega})$ as a subgroup of $G(T_{\Sym^{++}(3,\mathbb{R})})$
 (see \autoref{thm:G-GTOmega}).
In other words, 
 we shall see that any $g \in G(T_{\Omega})$
 can be 
 described by an element of $\Sp(6,\mathbb{R})$.
For $A \in H$,
 the corresponding $\rho(A) \in G(\Omega)$ is induced by the matrix 
 $\begin{pmatrix} A & 0 \\ 0 & A^{-T} \end{pmatrix}
 \in G(3, \real) \subset \Sp(6, \mathbb{R})$. 
For $v \in V$,
 we identify the translation $t_v : T_{\Omega} \owns z \mapsto z + v \in T_{\Omega}$ with
 the matrix
 $\begin{pmatrix} I & v \\ 0 & I \end{pmatrix} \in \Sp(6,\mathbb{R})$.
In this way,
 we regard $G_0:= \rho(H)$ 
 and $N^+$ as subgroups of $\Sp(6,\mathbb{R})$.
On the other hand,
 from a straightforward calculation,
 we see that the map $s$
  corresponds to the matrix
$$
 \begin{pmatrix} 
 0 & & & -1 & & \\ 
 & 0 & &  &-1 & \\
 & & 1 &  & & 0 \\
 1 & & &  0 & & \\
 & 1 & &  & 0 & \\
 & & 0 &  & & 1 
 \end{pmatrix} \in \Sp(6,\mathbb{R}).  
$$
Then an easy matrix calculation tells us that
 the transform $\tilde{t}_u = s t_u s^{-1} \in N^-$ corresponds to 
 $\begin{pmatrix} I & 0 \\ -u & I \end{pmatrix} \in \Sp(6,\mathbb{R})$.

Put $p_0 = i I \in T_{\Omega}$
 and 
 let $K=\left\{g \in G(T_{\Omega})\; |\; g \cdot p_0 = p_0\right\}$ be the isotropy subgroup 
  of $G(T_{\Omega})$
 at the point $p_0$.


\begin{lemma}[cf. \mbox{\cite[Lemma 4.1]{Gea87b}}] \label{lemma:isotropy}
One has
 $$
 K = \left\{
 k_{\theta, \phi} = \begin{pmatrix} C_{\theta,\phi} & - S_{\theta, \phi} \\ S_{\theta,\phi} & C_{\theta,\phi} \end{pmatrix} \; |\; 
 \theta,\phi \in [0,2\pi) \right\},
 $$
 where
 $$
 C_{\theta,\phi} = \begin{pmatrix} \cos \theta & & \\ & \cos \phi & \\ & & 1 \end{pmatrix}, \quad
 S_{\theta,\phi} = \begin{pmatrix} \sin \theta & & \\ & \sin \phi & \\ & & 0 \end{pmatrix}.
 $$
\end{lemma}


\begin{theorem} \label{thm:GTOmega}
The group $G(T_{\Omega})$ is generated by $G_0$,\, $N^+$ and $s$.
\end{theorem}
\begin{proof}
 Let us take any $g \in G(T_{\Omega})$
 and put $z = g \cdot p_0$.
Since $y = \Im z \in \Omega$,
 we can find $A \in H$ for which $\rho(A) \cdot I = y$.
Putting $x = \Re z \in V$,
 we have
 $g \cdot p_0 = z = t_x \rho(A) \cdot p_0$,
 so that $k = \rho(A)^{-1} t_x^{-1} g$ belongs to $K$.
Since $g = t_x \rho(A) k$,
 it is enough to show that $k$ is generated by $N^+$, $G_0$ and $s$.  

By \autoref{lemma:isotropy},
 we have $k = k_{\theta, \phi}$ with some $\theta,\phi \in [0,2\pi)$.
First we consider the case $\theta = \phi$.        
If $\theta \ne \frac{\pi}{2},\, \frac{3\pi}{2}$,
 then $\det C_{\theta, \theta} = \cos^2 \theta \ne 0$,
 and we have
$$
 k_{\theta, \theta} = t_v \rho(A) \tilde{t}_{-u} \in N^+ G_0 N^-
$$
 with
$$
 A =  (C_{\theta,\theta})^{-T} \in H, \quad 
 u = (C_{\theta,\theta})^{-1} S_{\theta,\theta} \in V', \quad
 v = - S_{\theta,\theta} (C_{\theta,\theta})^{-1} \in V
$$
 thanks to \autoref{lemma:HCdecomp_Sp6}.
Thus $k_{\theta,\theta}$ is generated by $N^+$, $G_0$ and $s$ in this case.
For the case that $\theta = \frac{\pi}{2}$ and $\theta = \frac{3\pi}{2}$,
 the element $k_{\theta,\theta}$ equals $s$ and $s^{-1}$ respectively,
 so that the claim holds for these cases, too.

Now we consider a general $k_{\theta,\phi}$.
We can take an appropriate $\alpha \in \mathbb{R}$ for which
 $\det C_{\theta + \alpha,\,\phi + \alpha} \ne 0$.
Similarly to the argument above,
 we see from \autoref{lemma:HCdecomp_Sp6} that
 $k_{\theta+\alpha, \phi+\alpha} \in N^+ G_0 N^-$.    
Finally,
 we have $k_{\theta,\phi} = k_{-\alpha,-\alpha} k_{\theta+\alpha, \phi+\alpha}$,
 which completes the proof.
\end{proof}

We remark that $G_0$ is not equal to the whole group $G(\Omega)$, 
 but is a subgroup of $G(\Omega)$ of index 2.
Indeed, $\rho(\sigma) \in G(\Omega) \setminus G_0$,
 and $\rho(\sigma)$ is a holomorphic automorphism on $T_{\Omega}$ but not an element of $G(T_{\Omega})$.
We also note that $G_0$ is not connected. 
Its identity component is $\rho(H^+)$.  
 
 \bigskip
 
Let us give another explicit description of the group $G(T_{\Omega})$
 as a subgroup of $\Sp(6,\mathbb{R})$.
We set
\begin{align*}
 W :&= 
 \left\{\begin{pmatrix} x_1 & 0 & x_6 \\ 0 & x_2 & x_7 \\ x_4 & x_5 & x_3\end{pmatrix}\; |\;
     x_1, \dots, x_7 \in \mathbb{R}\right\},\\
H' &:=
 \left\{\begin{pmatrix} a_1 & 0 & 0 \\ 0 & a_2 & 0 \\ a_4 & a_5 & a_3 \end{pmatrix}\; |\;
     a_1,\dots,a_5 \in \mathbb{R},\,a_3>0\right\},
\end{align*}
Then we have
\begin{align}
A,\,B \in H' &\Rightarrow A B \in H', \label{eqn:H-H}\\
A \in H',\,w \in W & \Rightarrow Aw,\,w A^T \in W, \label{eqn:H-W}\\
A \in H',\,u \in V' & \Rightarrow u A,\, A^T u \in V', \label{eqn:H-Vprime}
\end{align}
and
\begin{equation}\label{eqn:HWVprime}
A \in H',\,u \in V',\,w \in W \Rightarrow A + wu \in H'.
\end{equation}
Define
\begin{equation} \label{eqn:def_of_G}
 G := 
 \left\{ \begin{pmatrix} A & B \\ C & D \end{pmatrix} \in \Sp(6,\mathbb{R})\; |\; 
     A \in H',  \,B \in W,\,C \in V',\, D^T\in H'\right\}.
\end{equation}
Let us check that $G$ is a subgroup of $\Sp(6,\mathbb{R})$.
For two elements
 $g = \begin{pmatrix} A & B \\ C & D \end{pmatrix}$
 and  
 $g' = \begin{pmatrix} A' & B' \\ C' & D' \end{pmatrix}$
 of $G$,
 we have
$$
 g g' 
 = \begin{pmatrix} A A' + B C' & A B' + B D' \\ C A' + D C' & C B' + D D' \end{pmatrix}.
$$
Then
 we see from \autoref{eqn:H-H} -- \autoref{eqn:HWVprime} that
$$
 A A' + B C' \in H',\quad
 A B' + B D' \in W, \quad
 C A' + D C' \in V',\quad
 (C B' + D D')^T \in H',
$$
 so that $g g' \in G$.
On the other hand,
 since $G \subset \Sp(6,\mathbb{R})$,
 we have 
 $$
 g^{-1} 
 = \begin{pmatrix} 0 & I \\ -I & 0 \end{pmatrix} 
  g^T \begin{pmatrix} 0 & -I \\ I & 0 \end{pmatrix} 
 = \begin{pmatrix} D^T & -B^T \\ -C^T & A^T\end{pmatrix},
 $$
 for $g = \begin{pmatrix} A & B \\ C & D \end{pmatrix} \in G$,
 whence we see that $g^{-1} \in G$.


\begin{theorem} \label{thm:G-GTOmega}
The linear fractional action of $\Sp(6,\mathbb{R})$ on the Siegel upper half plane $T_{\Sym^{++}(3,\mathbb{R})}$
 induces an isomorphism from $G$ onto $G(T_{\Omega})$.
\end{theorem}
\begin{proof}
For $g = \begin{pmatrix} A & B \\ C & D \end{pmatrix} \in G$ and $z \in T_{\Omega}$,
 we obtain
 $A z + B \in W_{\mathbb{C}}$ by \autoref{eqn:H-W} and $(C z + D)^T \in H'_{\mathbb{C}}$ by \autoref{eqn:HWVprime},
 so that we have $g \cdot z \in W_{\mathbb{C}}$ by \autoref{eqn:H-W}.
On the other hand,
 since $g \in \Sp(6,\mathbb{R})$ and $z \in T_{\Sym^{++}(3,\mathbb{R})}$,
 we have $g \cdot z \in T_{\Sym^{++}(3,\mathbb{R})}$.
Thus $g \cdot z \in T_{\Omega} = T_{\Sym^{++}(3,\mathbb{R})} \cap W_{\mathbb{C}}$,
 and we have a group homomorphism from $G$ into $G(T_{\Omega})$.
Thanks to \autoref{thm:GTOmega},
 the map is surjective because
 $G$ contains the matrices corresponding to
 $t_v \in N^+\,\,\,(v \in V),\,
  \rho(A) \in G_0\,\,(A \in H)$
 and $s$.
Let us show the injectivity.
Take $g \in G$ such that
 $g \cdot z = z$ for all $z \in T_{\Omega}$.
Then $g \cdot p_0 = p_0$ together with $g \in \Sp(6,\mathbb{R})$ implies
 $g = \begin{pmatrix} A & -B \\ B & A \end{pmatrix}$ with $A + i B \in U(3)$.
Since $g \in G$, 
 we have $A \in H'$, $A^T\in H'$ and $B \in V'$.
Thus we get $A = C_{\theta,\phi}$ and $B = S_{\theta,\phi}$ with some $\theta, \phi \in [0,2\pi)$.
Let us consider $z \in T_{\Omega}$ with $z_1 = z_2 = i$.
Then
$$
 g \cdot
 \begin{pmatrix} i & 0 & z_4 \\ 0 & i & z_5 \\ z_4 & z_5 & z_3 \end{pmatrix}
 =
 \begin{pmatrix} i & 0 & e^{i \theta} z_4 \\ 0 & i & e^{i \phi} z_5 \\ 
 e^{i \theta} z_4 & e ^{i \phi} z_5 & z_3 + z_4^2 e^{i \theta} \sin \theta + z_5^2 e^{i \phi} \sin \phi \end{pmatrix}.
$$ 
Thus $g \cdot z = z$ implies $\theta = \phi = 0$, so that $g = I$.  
\end{proof}


We see from  \autoref{thm:G-GTOmega}  
 that
 each $g \in G(T_{\Omega})$ is uniquely extended to 
 a linear fractional transform on $T_{\Sym^{++}(3,\mathbb{R})}$.
Let us present one more description of the group $G$:


\begin{proposition} \label{prop:desc_G}
One has
$$
 G = 
 \left\{ \begin{pmatrix} A & B \\ C & D \end{pmatrix} \in \Sp(6,\mathbb{R})\; |\;
     A \in H',\, D^T \in H', \,D^T B \in V,\,\,C D^T\in V'\right\}.
$$
\end{proposition}
\begin{proof}
Let $G'$ be the right-hand side.
By \autoref{eqn:Sp6}, \autoref{eqn:H-W} and \autoref{eqn:H-Vprime}, we have $G \subset G'$.
To show the converse inclusion,
 we take $g \in G'$ with $\det D \ne 0$.
By \autoref{eqn:H-W} and \autoref{eqn:H-Vprime} again,
 we get
 $B = D^{-T} (D^T B) \in W$
 and
 $C = (C D^T)D^{-T} \in V'$.
Thus $g \in G$.
Since both $G$ and $G'$ are closed subset of $\mathrm{Mat}(6,\mathbb{R})$,
 we obtain $G' \subset G$ by a closure argument. 
\end{proof}

 
%
%
%
 
Let $\Upsilon \subset G(T_{\Omega})$ be the set consisting of $g \in G(T_{\Omega})$
 such that there exist $v \in V,\,\,A \in H$ and $u \in V'$ 
 for which
 $g = t_v \rho(A) \tilde{t}_u$.
Identifying $G(T_{\Omega})$ with $G$ by \autoref{thm:G-GTOmega}, 
 we get an explicit description of the set $\Upsilon$. 


\begin{proposition} \label{prop:Upsilon}
One has
 $$\Upsilon = \left\{\begin{pmatrix} A & B \\ C & D \end{pmatrix} \in G\; | \; \det D \not= 0\right\}.$$
Therefore  $\Upsilon$ is an open dense subset of $G(T_{\Omega})$.
\end{proposition}
\begin{proof}
Let $\Upsilon'$ be the right-hand side.
The inclusion $\Upsilon \subset \Upsilon'$ follows from \autoref{lemma:HCdecomp_Sp6}.
To show the converse inclusion,
 we take $g = \begin{pmatrix} A & B \\ C & D \end{pmatrix} \in \Upsilon'$.
Then we have the equality \autoref{eqn:HC_Sp6} and \autoref{eqn:Luv}.
In particular,
 $v = BD^{-1} \in \Sym(3,\mathbb{R})$ belongs to $W$ by \autoref{eqn:H-W},
 so that we get $v \in V = \Sym(3,\mathbb{R}) \cap W$.
On the other hand,
 we have
 $u= D^{-1}C \in V'$ by \autoref{eqn:H-Vprime}
 and $L = D^{-T} \in H$. 
Thus $g = t_v \rho(L) \tilde{t}_{-u} \in \Upsilon$ and 
 the assertion is verified. 
\end{proof}
 \section{The triple decomposition of   $\Gamma$}
 We shall investigate decomposition structures of the compression semigroup $\Gamma$
 defined by \autoref{eqn:def_of_Gamma}. 
More precisely, 
 we will prove that any element $g$ of the semigroup $\Gamma$ admits 
 a \textit{triple decomposition}
 $g = t_v \rho(A) \tilde{t}_{-u}$, which is unique by \autoref{lemma:HCdecomp_Sp6}.
 
Consider the following two closed subsemigroups of $\Gamma$
 \begin{eqnarray*}
 \Gamma^+&:=& \left\{t_v\,|\,v \in \overline{\Omega}\right\}=\left\{\begin{pmatrix}I&v\\ 0&I\end{pmatrix}\; |\; v\in\overline{\Omega}\right\},\\
 \Gamma^-&:=& \left\{\tilde{t}_{-u}\,|\,u \in \overline{\Omega}\right\}=\left\{\begin{pmatrix}I&0\\ u&I\end{pmatrix}\; |\; u\in\overline{\Omega}\cap V'\right\},
 \end{eqnarray*}
 and
 \begin{eqnarray*}
 {\Gamma^+}^0&:=& \left\{t_v\,|\,v \in {\Omega}\right\}=\left\{\begin{pmatrix}I&v\\ 0&I\end{pmatrix}\; |\; v\in {\Omega}\right\},\\
 {\Gamma^-}^0&:=& \left\{\tilde{t}_{-u}\,|\,u \in {\Omega}\right\}=\left\{\begin{pmatrix}I&0\\ u&I\end{pmatrix}\; |\; u\in {\Omega}\cap V'\right\}.
 \end{eqnarray*}
The latest  are two  subsemigroups of the interior $\Gamma^0$ of $\Gamma$.
Now we state our first main theorem.


\begin{theorem} \label{thm:Gamma}
The semigroup $\Gamma$ is contained in $\Upsilon$.
Moreover, 
 one has $\Gamma = \Gamma^+ G_0 \Gamma^-$. 
\end{theorem}
\begin{proof}
First we observe that,
 for $g = \begin{pmatrix} A & B \\ C & D \end{pmatrix} \in G$
with
\begin{eqnarray*}
 &&A = \begin{pmatrix} a_1 & & \\ 0 & a_2 & \\ a_4 & a_5 & a_3 \end{pmatrix}, \quad 
B= \begin{pmatrix} b_1 & 0 & b_6 \\ 0 & b_2 & b_7 \\ b_4 & b_5 & b_3 \end{pmatrix}, \\
&&C= \begin{pmatrix} c_1 & 0& 0 \\ 0 & c_2 & 0 \\ 0 & 0 & 0 \end{pmatrix}, \quad
D = \begin{pmatrix} d_1 & 0 & d_4 \\ & d_2 & d_5 \\ & & d_3 \end{pmatrix}, 
\end{eqnarray*}
 the equality 
 $A D^T - B C^T = I$
 implies
\begin{equation} \label{eqn:ad-bc}
 a_k d_k - b_k c_k = 1 \qquad (k=1,\,2).
\end{equation}
Moreover,
 if $z' = g \cdot z \in V_{\mathbb{C}}$ with $z \in V_{\mathbb{C}}$,
 then 
\begin{equation} \label{eqn:zk}
 z'_k = \frac{a_k z_k + b_k}{c_k z_k + d_k} \qquad(k=1,\,2).
\end{equation}
Now we suppose $g \not\in \Upsilon$,
 which means $\det D = 0$ by  \autoref{prop:Upsilon}.
Since $D^T \in H'$,
 we have $d_1= 0$ or $d_2 = 0$.
If $d_1 = 0$,
 we have $c_1 = -\frac{1}{b_1} \ne 0$ by \autoref{eqn:ad-bc}.
Let us consider the case 
$$
 z = \begin{pmatrix} x_1 & 0 & 0 \\ 0 & 1 & 0 \\ 0 & 0 & 1 \end{pmatrix} \in V.
$$
By \autoref{eqn:zk},
 we have
 $z'_1 = -a_1b_1 - \frac{b_1^2}{x_1}$,
 so that
 we can take $x_1>0$ for which $z'_1 <0$.
Then $z \in \Omega$ and $z' = g \cdot z \notin \Omega$,
 which imply that $g \not\in \Gamma$.
Similarly we can show $g \notin \Gamma$ if $d_2 = 0$.
Therefore
 we conclude that $\Gamma \subset \Upsilon$.
 
Now take $g \in \Gamma$
 and let $g = t_v \rho(A) \tilde{t}_{-u}\,\,\,(u \in V',\,\,A \in H,\,\,v \in V)$
 be a triple decomposition.
Let $\{x_n\}_{n=1}^{\infty} \subset \Omega$ be a sequence converging to $0$.
Then $\Omega \owns g \cdot x_n = v + \rho(A) \tilde{t}_{-u} x_n \to v$ as $n \to \infty$,
 so that we get $v \in \overline{\Omega}$.
Thanks to \autoref{eqn:zk},
 we have
 $z'_k = v_k + \frac{a_k^2 z_k}{z_k+u_k}\,\,\,(k=1,2)$
 for $z' = g \cdot z \in V_{\mathbb{C}}$ with $z \in V_{\mathbb{C}}$.
In particular,
 if $u_1 <0$, then $g \cdot x$ is not defined for
 $$
 x = \begin{pmatrix} -u_1 & 0 & 0 \\ 0 & 1 & 0 \\ 0 & 0 & 1 \end{pmatrix} \in \Omega,
 $$
 which contradicts $g \in \Gamma$.
Therefore $u_1 \ge 0$.
We see that $u_2 \ge 0$ similarly, 
 which completes the proof of the theorem. 
\end{proof}

As a consequence, we have
 \begin{equation}\label{semipg-trace}
 \Gamma_{\Sp}\cap G=\Gamma \;\;\text{and}\;\;  \Gamma_{\Sp}^0\cap G=\Gamma^0.
 \end{equation}

Let us describe $\Gamma$ in a matrix block form.

\begin{proposition} \label{prop:desc_Gamma}
One has
$$
 \Gamma = 
 \left\{\begin{pmatrix} A & B \\ C & D \end{pmatrix} \in G\; |\;
     \det(D) \not= 0,\,\,D^TB \in \overline{\Omega},\,\,C D^T \in \overline{\Omega} \cap V'\right\}.
$$
\end{proposition}
\begin{proof}
Let $g = t_v \rho(L) \tilde{t}_{-u}$ be the triple 
 decomposition of $g \in \Gamma$.
Thanks to \autoref{eqn:Luv}, 
 we have $u = D^{-1}C = D^{-1} (C D^T) D^{-T}$
 and $v = B D^{-1} = D^{-T}(D^TB) D^{-1}$.
Therefore the assertion follows from \autoref{thm:Gamma}.  
\end{proof}

%
%
%
%
%
\section{The Ol'shanski\u{\i} polar decomposition of $\Gamma$}
 
We see from (\autoref{eqn:def_of_G}) that
 the Lie algebra $\mathfrak{g}$ of $G$ equals the subalgebra of $\mathfrak{sp}(6,\real)$ 
 given by
$$
 \mathfrak{g} = \set{\begin{pmatrix} A & v \\ u & -A^T \end{pmatrix}}
{A \in \mathfrak{h},\,\, v \in V,\,\, u \in V'}, 
$$
 where $\mathfrak{h}$ is the Lie algebra of $H \subset \mathrm{GL}(3, \real)$, 
 that is,
$$ \mathfrak{h} = \set{ \begin{pmatrix} a_1 & & \\ 0 & a_2 & \\ a_4 & a_5 & a_3 \end{pmatrix}}
                 {a_1, \dots, a_5 \in \real}. $$
Then $\mathfrak{g}$ is graded by $\mathrm{ad}(Z_0)$ with 
 $Z_0 := \begin{pmatrix} I/2 & 0 \\ 0 & - I/2 \end{pmatrix} \in \mathfrak{g}$.
Namely, 
 if $\mathfrak{g}_k := \set{X \in \mathfrak{g}}{[Z_0, X] = k X}$, 
 then $\mathfrak{g} = \mathfrak{g}_{-1} \oplus \mathfrak{g}_0 \oplus \mathfrak{g}_1$
 with
\begin{eqnarray*}
 \mathfrak{g}_{-1}& =& \set{\begin{pmatrix} 0 & 0 \\ u & 0 \end{pmatrix}}{u \in V'}, \\
 \mathfrak{g}_0 &=& \set{\begin{pmatrix} A & 0 \\ 0 & - A^T \end{pmatrix}}{A \in \mathfrak{h}}, \\
\mathfrak{g}_1 &=& \set{\begin{pmatrix} 0 & v \\ 0 & 0 \end{pmatrix}}{v \in V}.
\end{eqnarray*}

Let
$$C := \left\{\begin{pmatrix} 0& v \\ u &0\end{pmatrix}\; |\; 
 v\in \overline{\Omega},\,  u\in\overline{\Omega}\cap V'\right\}.$$
Then $C$ is an $Ad(G(\Omega))$-invariant closed convex cone 
 in $\frak{g}_{-1}\oplus\frak{g}_{1}$ 
 which is proper ($C\cap-C=\{0\}$) and  generating ($C^0\not=\emptyset$). 
Its interior $C^0$ is the set of 
 matrices with $v\in\Omega$ and $u\in\Omega\cap V'$.

\begin{theorem}
The compression semigroup $\Gamma$ has the following Ol'shanski\u{\i} polar decomposition
$$\Gamma=G_0 \exp(C)$$
 with $\Gamma^0=G_0\exp(C^0)$ as interior.
\end{theorem}
\begin{proof} 
Let us denote $S := G_0\exp(C)$ and prove that $\Gamma=S$.

First by \cite{Olsh82}, \cite{Law94} 
 it follows that $S$ is a closed subsemigroup of $G$. 
Further,
 it is clear that $\Gamma^+=\exp(\bar{\Omega}) $, 
 $\Gamma^-=\exp(\bar{\Omega}\cap V')$ and $G_0$ 
 are closed subsemigroups of $S$. 
Thus we have $\Gamma\subset S$ by \autoref{thm:Gamma}. 

On the other hand, 
 since $G_0$ and $\exp(C)$ are subsemigroups of $\Gamma_{\mathrm{Sp}}$, 
 we see that $G_0\exp(C)\subset \Gamma_{\mathrm{Sp}}$. 
In addition, $G_0\subset G$ and $\exp(C)\subset G$, 
so that $S=G_0\exp(C)$ is contained in both $G$ and $\Gamma_{\Sp}$.
Therefore  $S \subset \Gamma$ thanks to \autoref{semipg-trace}.
\end{proof}

\section{A counter-example to the contraction property of $\Gamma$}
On a proper open convex cone $\Cone \subset \real^n$, 
 the second derivative of the logarithm of the K\"ocher-Vinberg characteristic function 
 $\varphi_{\Cone}$ of $\Cone$ gives a canonical Riemannian metric
 (see \cite[Section I.4]{F-K}, \cite[Chapter I, Section 3]{Vin63}):
 $$ (v|v')_x := D_v D_{v'} \log \varphi_{\Cone}(x) \qquad (v,v' \in \real^n,\,\, x \in \Cone),$$
 where $D_v$ denotes the directional derivative in $v$.
Thanks to the relative invariance of $\varphi_{\Cone}$ 
 under the action of the linear automorphism group $G(\Cone)$
 (see (\autoref{eqn:phi-invariance})),
 the canonical metric is $G(\Cone)$-invariant.
In particular, 
 if $\Cone$ is a symmetric cone, 
 the metric makes $\Cone$  a Riemannian symmetric space. 
For example, 
 if $\Cone = \mathrm{Sym}^{++}(3, \mathbb{R})$,
 then the metric is given by the formula
\begin{equation} \label{eqn;metric_on_Sym3}
 (v|v')_x = 2 \mathrm{tr}\,(x^{-1} v x^{-1} v') \qquad 
 (v,v' \in \mathrm{Sym}(3, \mathbb{R}),\, x \in \mathrm{Sym}^{++}(3, \real)). 
\end{equation}  
It is proved in \cite[Section 5]{Kou95} that,
 if $\Cone$ is symmetric,  
 the compression semigroup $\Gamma_{\Cone}$ of $\Cone$ has the contraction property
 with respect to the canonical metric, that is,
$$
 (J(g,x)v| J(g,x)v)_{g(x)} \le (v|v)_x \qquad (g \in \Gamma_{\Cone}, x \in \Cone,\,\, v \in \real^n), 
$$
 where $J(g,x)$ stands for the Jacobi matrix of $g$ at $x$.
We shall see that it is no longer the case when $\Cone$ is the dual Vinberg cone $\Omega$.

Recalling \autoref{eqn:phi-det},
 we see that the canonical Riemannian metric on $\Omega$ is given by
\begin{equation} \label{eqn:metric_on_Omega}
(v|v')_x = -\frac{1}{2} \Bigl( \frac{v_1 v'_1}{x_1^2} + \frac{v_2 v'_2}{x_2^2}\Bigr) 
+ 2\, \mathrm{tr}\,(x^{-1} v x^{-1} v' )
 \qquad (v,v' \in V,\,\, x \in \Omega).
\end{equation}
Now we consider the case $v = v' = \begin{pmatrix} 1 & 0 & 1 \\ 0 & 0 &  0 \\ 1 & 0 & 1 \end{pmatrix}$ and $x = I_3$. \\
Then 
$ (v|v)_x = -\frac{1}{2} + 2 \times 4 = 7.5$. 
In view of \autoref{thm:Gamma}, 
 we put $g_0 := t_{v_0} \in \Gamma$ with 
 $v_0 := \begin{pmatrix} 1 & 0 & -1 \\ 0 & 1 & 0 \\ -1 & 0 & 1.01 \end{pmatrix} \in \Omega$.
Then 
 $g_0(x) = I_n + v_0 = \begin{pmatrix} 2 & 0 & -1 \\ 0 & 2 & 0 \\ -1 & 0 & 2.01 \end{pmatrix}
 \in \Omega$ and
 $J(g_0, x)v = v$ since $g_0$ is a translation.
We observe
$$
 (J(g_0,x)v| J(g_0, x)v)_{g_0(x)} = -\frac{1}{8} + 2 \Bigl( \frac{6.01}{3.02} \Bigr)^2 
= 7.795\cdots > 7.5 = (v|v)_x. 
$$ 
This phenomenon is caused by a behavior of 
 `the extra term' $-\frac{1}{2}(\frac{v_1 v'_1}{x_1^2} + \frac{v_2 v'_2}{x_2^2})$
 in \autoref{eqn:metric_on_Omega}, 
 compared with \autoref{eqn;metric_on_Sym3}.
Actually, the decrease of the second term $2\,\mathrm{tr} (x^{-1} v x^{-1} v')$ is little
 (from $8$ to $2 \Bigl( \frac{6.01}{3.02} \Bigr)^2 =  7.920\cdots$), 
 while the extra term increases from $-1/2$ to $-1/8$.

\end{document}